\RequirePackage{amsmath}
\documentclass[runningheads]{llncs}
\usepackage[T1]{fontenc}
\usepackage{graphicx}
\usepackage{color}
\usepackage{multirow}
\usepackage{charter}

\usepackage[backgroundcolor=SkyBlue,linecolor=SkyBlue,obeyFinal]{todonotes}

\makeatletter
\providecommand\@dotsep{5}
\def\listtodoname{}
\def\listoftodos{\@starttoc{tdo}\listtodoname}
\makeatother

\newcounter{nmcomment}

\usepackage{parskip}
\usepackage{hyperref}
\hypersetup{colorlinks,allcolors=black}
\usepackage{comment}
\excludecomment{shortver}
\includecomment{longver}

\usepackage{wrapfig}
\usepackage{tikz}
\usetikzlibrary{positioning}
\usetikzlibrary{arrows}
\usetikzlibrary{decorations.pathmorphing}
\usetikzlibrary{decorations.markings}

\usepackage{latexsym,amssymb,amsmath,mathtools,eulervm,xspace}

\begin{document}
\title{Algorithms for the Minimum Generating Set Problem}

%
%
\author{Bireswar Das \and Dhara Thakkar }
\authorrunning{B. Das and D. Thakkar}
%
\institute{Indian Institute of Technology Gandhinagar, Gandhinagar 
\email{\{bireswar,thakkar\_dhara\}@iitgn.ac.in}
}

\maketitle              
\begin{abstract}
For a finite group $G$, a generating set of minimum size is called \emph{a minimum generating set} of $G$. The size of a minimum generating set of $G$ is denoted by $d(G)$. Given a finite group $G$ and an integer $k$, deciding if $d(G)\leq k$ is known as the \emph{minimum generating set} (MIN-GEN) problem. 

A group $G$ of order $n$ has generating set of size $\lceil \log_p n \rceil$ where $p$ is the smallest prime dividing $n=|G|$. This fact is used to design an $n^{\log_p n+O(1)}$-time algorithm for the group isomorphism problem of groups specified by their Cayley tables (attributed to Tarjan by Miller, 1978). The same fact can be used to give an $n^{\log_p n+O(1)}$-time algorithm for the MIN-GEN problem. We show that the MIN-GEN problem can be solved in time $n^{(1/4)\log_p n+O(1)}$ for general groups given by their Cayley tables. This runtime incidentally matches with the runtime of the best known algorithm for the group isomorphism problem.

 We show that if a group $G$, given by its Cayley table, is the product of simple groups then a minimum generating set of $G$ can be computed in time polynomial in the order of $G$.

Given groups $G_i$ along with $d(G_i)$ for $i\in [r]$ the problem of computing $d(\Pi_{i\in[r]} G_i)$ is nontrivial. As a consequence of our result for  products of simple groups we show that this problem also can be solved in polynomial time for Cayley table representation.

For the MIN-GEN problem for permutation groups, to the best of our knowledge, no significantly better algorithm than the brute force algorithm is known. For an input group $G\leq S_n$, the brute force algorithm runs in time $|G|^{O(n)}$ which can be $2^{\Omega(n^2)}$. We show that if $G\leq S_n$ is a primitive permutation group then the MIN-GEN problem can be solved in time quasi-polynomial in $n$.

We also design a $\mathrm{DTIME}(2^n)$ algorithm for computing a minimum generating set of permutation groups all of whose non-abelian chief factors have bounded orders.


\end{abstract}

\section{Introduction}\label{introduction}

Let $G$ be a finite group. A generating set of $G$ with minimum size is called \emph{a minimum generating set}. The size of a minimum generating set of a group $G$ is denoted by $d(G)$.

In this paper, we consider the problem of computing $d(G)$, computing a minimum generating set of a given group $G$, and a decision version of the problem denoted MIN-GEN. The input to the MIN-GEN problem is a finite group $G$ and an integer $k$, and the task is to decide if $d(G) \leq k$.

Papadimitriou and Yannakakis defined an analogous version of the MIN-GEN problem for quasigroups given by their Cayley tables \cite{papadimitriou}. They proved that the problem is complete for $\beta_2 \mathrm{P}$ (it is the class of all problems in $\mathrm{NP}$ that need $O(\log^2 n)$ nondeterministic bits. It is also denoted as $\mathrm{NP}(\log^2n)$). Later Arvind and Tor\'an prove that the problem is in $\mathrm{DSPACE}(\log^2 n)$ \cite{Arvind}. Arvind and Tor\'an gave a polynomial-time algorithm for the MIN-GEN problem for nilpotent groups given by their Cayley tables \cite{Arvind}.

Various structural and quantitative properties related to the minimum generating set problem has been studied before the above-mentioned results, mainly from a mathematical perspective. Gasch\"utz in 1959 studied the problem and provided some of the key ideas to solve the problem for solvable groups \cite{gaschutz1959mingensolvable}. Based on Gasch\"utz's ideas, Lucchini and Menegazzo designed two algorithms to solve the problem for solvable groups when the input group is given by a polycyclic representation \cite{lucchini-minimal-solvable}. The first algorithm works when a chief series is given as input. The second algorithm is also similar but uses the derived series. Lucchini and Menegazzo implemented these algorithms and tested the performance of those implementations. While explicit runtime analyses of their algorithms are not presented in their paper, it is not hard to see that if we use existing polynomial-time algorithms for some of the subroutines they use (computing a chief series, computing a minimal normal subgroup), then their algorithms actually run in polynomial time for solvable permutation groups. For the sake of completeness, we present a simple recursive algorithm to solve the MIN-GEN problem for solvable groups in Section~\ref{MIN-GENProblem-Solvable-PG} based on the ideas by Gasch\"utz \cite{gaschutz1956}, and Lucchini and Menegazzo \cite{lucchini-minimal-solvable}. 


\subsection{A Faster Algorithm for MIN-GEN}

Given a group $G$ by its Cayley table, it is easy to design an algorithm to solve the MIN-GEN problem in time $n^{\log_p n+O(1)}$, where $n=|G|$ and $p$ is the smallest prime factor of $n$ \cite{Arvind}. The algorithm basically tries all possible subsets of $G$ of size at most $\lceil \log_pn \rceil$. The correctness of the algorithm follows from the fact that any group of order $n$ has a generating set of size at most $\log_pn$ where $p$ is the smallest prime dividing $n$. We ask the question if we can obtain an algorithm for MIN-GEN that runs in time $n^{c\log_pn+O(1)}$ for $c < 1$. We note that reducing the constant factor in the exponent can sometimes be an interesting and challenging problem. One example of such a problem is the group isomorphism problem (GpI) \cite{liptonsBlog}.   

Given two groups by their Cayley table, the group isomorphism problem (GpI) is to decide if they are isomorphic. There is an $n^{\log_{p} n+O(1)}$-time algorithm, known as the \emph{generator-enumerator algorithm}, that solves the group isomorphism problem when $n$ is the order of each input group, and $p$ is the smallest prime dividing $n$. The algorithm is attributed to Tarjan by Miller \cite{Miller}. The first step in the generator-enumerator algorithm is similar to the naive algorithm for MIN-GEN mentioned above: It is to pick a generating set of one of the groups of size $O(\log_pn)$ in a brute-force manner. 

Over the past years, there has been significant progress in improving the exponent in the running time. For $p$-groups, Rosenbaum and Wagner gave an $n^{(1/2) \log_{p} n+O(1)}$-time algorithm \cite{RosenbaumWagener-pGroup}. Rosenbaum    \cite{rosenbaum2013breaking} gave an algorithm to test isomorphism of solvable groups that runs in time $n^{(1/2) \log_{p} n+O(\log n/ \log \log n)}$.  Later they improved the runtime and gave an $n^{(1/4) \log_{p} n+O(\log n/ \log \log n)}$-time algorithm for the solvable group isomorphism problem \cite{RosenbaumBidirectional}. In the same paper, they presented an $n^{(1/2) \log_{p} n+O(1)}$-time isomorphism algorithm for  general groups \cite{RosenbaumBidirectional}. Le Gall and Rosenbaum pointed out that by combining the techniques of Luks \cite{luks2015group} with Rosenbaum's bidirectional collision testing \cite{RosenbaumBidirectional} it is possible to design an $n^{(1/4) \log_{p} n+O(1)}$ algorithm for GpI for general groups   \cite{LeGall-Rosenbaum}.

We design an $n^{(1/4) \log_{p} n+O(1)}$ algorithm for the MIN-GEN problem when the group is given by its Cayley table representation. While it is not clear what the relative complexity of GpI and MIN-GEN is, the runtime of our algorithm for MIN-GEN incidentally matches with the best known algorithm for the group isomorphism problem. We prove this result in Section~\ref{MIN-GEN-General-Groups}.



\subsection{Products of simple groups and product of groups}\label{Product-Simple-Intro}

Simple groups play an important role in group theory. It is known that any simple group can be generated by at most $2$ elements, and one can design a polynomial-time algorithm for computing a minimum generating set of simple groups given by their Cayley tables. How about the product of simple groups? The growth in the size of the minimum generating set of the product of simple groups as we increase the number of simple groups in the product is very interesting. For example, $A_5$, which is a simple group, is generated by two elements. It turns out that $d(A_5^{i}) = 2$ for $i \leq 19$ \cite{Phall1936}. It is also fascinating to note that $d(A_5^{6464040})=5$ \cite{Wiegold-1}. It is not only this curious behavior of the product of simple groups that makes the minimum generating set problem for such groups interesting, but it is also because of the fact that a minimal normal subgroup of a group is a product of simple groups. Furthermore, it is  well known that minimal normal subgroups play an important role in the theory of minimum generating set \cite{lucchini-Generators-and-minimal-normal,lucchini-uniuqe-minimal}. We show that if $G$, given by its Cayley table, is a product of simple groups then a minimum generating set of $G$ can be computed in polynomial time.

By Remak-Krull-Schmidt theorem, we know that every group is an internal direct product of indecomposable subgroups (see e.g., \cite{hungerfordAlgebra}). Therefore, a natural approach to solve the MIN-GEN problem for general group $G$ would be to compute minimum generating sets for the indecomposable subgroups  and next to use the computed information to solve the problem. This motivated the following problem:

Given $G=G_1 \times G_2$ along with $d(G_1)$ and $d(G_2)$ (i.e., $d(G_1)$ and $d(G_2)$ are also given as an input) can we find $d(G)$? As a consequence of our result for the product of simple group, we show that this problem can be solved in polynomial time for Cayley table representation. In other words, we show that the MIN-GEN problem for general groups is polynomial time Turing reducible to the MIN-GEN problem for indecomposable groups (Section~\ref{MIN-GEN-Product-Simple-Groups}). 

\subsection{Menegazzo's, and Cameron's questions}
For a general permutation group $G \leq S_n$, a minimum generating set of $G$ can be computed by trying all possible subsets of $G$ of size $(n-1)$. Note that by Jerrum's filter argument every permutation group $G \leq S_n$ has a generating set of size $(n-1)$ \footnote{This can be improved to $n/2$ (see e.g., \cite{Man-Lucchini-PG-GenSet}).} \cite{jerrum1986compact}. Thus, the runtime of this naive algorithm is $|G|^{O(n)}$ which can be $2^{\Omega(n^2)}$ as $|G|$ could be as large as $n!$. To the best of our knowledge, no significantly better algorithm is known for the MIN-GEN problem for permutation groups. 

In a survey paper on the minimum generating set \cite{Survey-Menegazzo}, Menegazzo asked the following two questions:

\textsl{``Problem 1: Give an algorithm to transform a given set of generators into a generating set of minimum cardinality, for permutation and linear groups.''}

\textsl{``Problem 2: Give an algorithm to find a set of generators of the expected cardinality, for particular classes of permutation and linear groups (e.g., $\log n$ for primitive subgroups of $S_n$, etc.).''}

Cameron also asked the following similar question on his list of problems on permutation groups   \cite{cameron-open-problem}:

\textsl{``Problem 24: Find an efficient algorithm (e.g., an on-line algorithm) for finding a generating set of size at most $n/2$ for the subgroup generated by an arbitrary set of permutations.''}

We first study the MIN-GEN problem for primitive subgroups of $S_n$. Primitive groups are not only important from a group theoretic perspective (see e.g., \cite{dixon-permutation}), they have also played a crucial role in the design of efficient algorithms for the graph isomorphism problem \cite{luks1982isomorphism,babai2016graph}. We design an algorithm that takes a primitive group $G \leq S_n$ as input and outputs $d(G)$ in quasi-polynomial time in $n$ (Section~\ref{MIN-GEN-Primitive-PG}).

The next class of groups we consider are permutation groups, all of whose non-abelian chief factors are of order at most $l$. We denote this class as $\chi(l)$. This is a superclass of solvable groups. Our initial motivation was to study the minimum generating set problem for the class $\Gamma(l)$ of permutation groups, all of whose non-abelian composition factors are isomorphic to subgroups of $S_l$. Note that for fixed $l$, the orders of the non abelian composition factors of groups in $\Gamma(l)$ are bounded. The structure of primitive groups in $\Gamma(l)$ has been studied by Babai, Cameron, and P\'alfy \cite{babai-palfy-cameron}. While we do not know any non-trivial results on the MIN-GEN problem for groups in $\Gamma(l)$, for the class $\chi(l)$, we design an $O(l^n n^{O(1)})$ algorithm for the minimum generating set problem where $n$ is the degree of the input group. For transitive groups in $\chi(l)$, a similar technique can be used to give a sub-exponential time algorithm for the MIN-GEN problem (Section~\ref{Min-Gen-Bounded-non-abelain-chief-factor}).



\section{Group Theory Background}\label{Preliminary}

In this section, we recall some relevant definitions from group theory. An interested reader may refer to standard books for more details (see e.g., \cite{rotman,marshall-hall-theory,dixon-permutation}). 

We consider groups with finitely many elements. The \emph{order} of group $G$ is the number of elements in a group $G$, denoted by $|G|$. A subset $A ( \neq \emptyset)$ of a group $G$ is called a \emph{generating set} of $G$ if every element of $G$ can be written as a finite product of elements of $A$ and their inverses. A generating set $A$ with minimum size is called a \emph{minimum generating} set of $G$. The size of a minimum generating set is denoted by $d(G)$. By convention, $d(\{e\})=0$.

Let $H$ be a subgroup of $G$, if for all $a \in G$, $a^{-1}Ha=H$ then we say that $H$ is a \emph{normal subgroup} of $G$. A group $G$ is called \emph{simple} if it has no proper nontrivial normal subgroup. A normal subgroup $H\neq 1$ of $G$ is called \emph{minimal normal} if the only normal subgroups of $G$ contained in $H$ are $1$ and $H$. A minimal normal subgroup $H$ of a finite group $G$ is either simple or a direct product of isomorphic simple groups (see e.g., page 106, \cite{rotman}). The \emph{normal closure} of a subset $S$ of a group $G$ is a normal subgroup generated by $\{g^{-1}sg | g \in G \}$ of all conjugates of elements of $S$ in $G$, denoted by $\langle S^G \rangle$. 

A group $G$ is called \emph{indecomposable group} if $G\neq \{e\}$ and $G$ is not the (internal) direct product of two of its proper subgroups.  The Remak-Krull-Schmidt theorem says that any finite group can be factored as a direct product of indecomposable groups \cite{hungerfordAlgebra}.

Subgroups of $S_n$ are known as \emph{permutation groups} of \emph{degree} $n$. A group $G \leq S_n$ is said to be \emph{transitive} if $1^{G}=[n]$. Let $G_{\alpha}:=\{x\in G \,|\, \alpha^{x}=\alpha \}$. A transitive group $G$ on a set $[n]$ is said to act \emph{regularly} if $G_{\alpha}=\{ e\}$ for all $\alpha \in [n]$. For a subset $\Gamma \subset [n]$, $\Gamma^x=\{x(\alpha) \,|\, \alpha \in \Gamma \}$. A nonempty subset $\Gamma$ of $[n]$ is called a \emph{block} for $G$ if for each $x \in G$ either $\Gamma^x = \Gamma$ or $\Gamma^x \cap \Gamma = \emptyset$. The set $[n]$ and $\{\alpha\}$ ($\alpha \in [n]$) are called trivial blocks. Any other blocks is called nontrivial. A transitive group $G \leq S_n$ is \emph{primitive} if $G$ has no nontrivial blocks on $[n]$. The \emph{socle} of a group $G$ is the subgroup generated by the set of all minimal normal subgroups of $G$; it is denoted by $ \mathrm{soc}(G)$.


Let $K$ and $H$ be groups and suppose $H$ acts on the nonempty set $\Gamma$. Let $\mathrm{Fun}(\Gamma, K)$ be the set of all functions from $\Gamma$ to $K$. Then the \emph{wreath product} of $K$ by $H$ (denoted by $K \wr_{\Gamma} H$) with respect to this action is defined to be the semidirect product $\mathrm{Fun}(\Gamma, K) \rtimes H$ where $H$ acts on the group $\mathrm{Fun}(\Gamma, K)$ via $f^x(\gamma):=f(\gamma^{x^{-1}})$ for all $f \in \mathrm{Fun}(\Gamma, K)$, $\gamma \in \Gamma$ and $x \in H$.

\begin{definition}(see e.g., section 8.4, \cite{marshall-hall-theory})
In a group $G$, a sequence of subgroups
$$ 1= G_{m} \trianglelefteq  G_{m-1} \trianglelefteq  \dots \trianglelefteq G_1 \trianglelefteq   G_{0}=G$$ is called a \emph{composition series} of $G$ if $G_{i-1} / G_{i}$ is simple for all $1 \leq i \leq m$.
\end{definition}

\begin{definition}(see e.g., section 9.2, \cite{marshall-hall-theory})
A group $G$ is \emph{solvable} if and only if it has a composition series all of whose factors are cyclic groups of prime order. 
\end{definition}

\begin{definition}(see e.g., section 8.4, \cite{marshall-hall-theory})
A \emph{chief series} in a group $G$ is a sequence of subgroups 
$$ 1= G_{m} \trianglelefteq  G_{m-1} \trianglelefteq  \dots \trianglelefteq  G_1   \trianglelefteq  G_{0}=G$$ such that for all $ 1 \leq i \leq m$, $G_{i} \trianglelefteq G$ and each factor $G_{i-1} / G_{i}$ is a minimal normal subgroup of $G / G_i$.
\end{definition}

\begin{theorem}(Jordan-Hölder Theorem)
\noindent Let $G$ be a finite group with $G \neq 1 $ and,
$1=N_{r} \trianglelefteq N_{r-1} \trianglelefteq \cdots \trianglelefteq N_{1} \trianglelefteq N_{0}=G$ and $1=M_{s} \trianglelefteq M_{s-1}\trianglelefteq \cdots \trianglelefteq M_{1} \trianglelefteq M_{0}=G $ are two chief series for $G$, then $r=s$ and there is some permutation $\pi \in S_{r}$ such that, $$ \frac{N_{\pi(i)-1}}{N_{\pi(i)}} \cong \frac{M_{i-1}}{M_{i}}, \quad \text{for} \,\, 1 \leq i \leq r. $$
\end{theorem}

\begin{remark}
The common version of the Jordan-Hölder theorem is stated in terms of the composition series. However, the above theorem is also called the Jordan-Hölder Theorem (see e.g., Theorem 8.4.4,  \cite{marshall-hall-theory}).
\end{remark}

\section{Results on Extending Generating Sets}\label{Meta-Theorem}
In this section, we state some results that we use in the later sections to prove the main theorems. Each of these results provides ways of finding a generating set of a group from the generating set of a quotient of the group. The procedures described in Lemma \ref{Gen-Set-Quotient} and Lemma \ref{Gen-Set-Quotient-Cayley} serve as subroutines for some of the results that we later prove. These two lemmas are also important to avoid taking multiple levels of quotients (e.g., a quotient of quotient group, etc.) in the recursive algorithms that we design.    

\begin{theorem}\cite{lucchini-Generators-and-minimal-normal}\label{bound-on-d(G)}
If $G$ is a finite group and $N$ is a minimal normal subgroup of G, then $d(G)\leq max(2,d(G/N)+1)$. In particular, $d(G/N)\leq d(G) \leq d(G/N)+1$.
\end{theorem}

It is standard fact that if $G$ is solvable then its minimal normal subgroups are abelian. In the Appendix (Theorem \ref{bound-on-d(G)-abelian-case}) we provide a proof of the above theorem for the case when $N$ is abelian for the sake of completeness of the algorithm for computing minimum generating set of  solvable permutation groups (Section~\ref{MIN-GENProblem-Solvable-PG}).

The proof of the case when $N \neq G$ and $N$ is non-abelian is non-trivial. However, we mention that in this case there exists $x_{1},\ldots,x_{t+1} \in N$ such that either $\{g_{1}x_{1},\ldots,g_{t}x_{t},x_{t+1}\}$ generates $G$ or $\{ x_{1} g_{1},g_{2},\ldots,g_{t},x_{2}\}$ generates $G$ given that $G/N=\langle g_{1}N,\ldots,g_{t}N\rangle$. An interested reader may refer to the paper by Lucchini \cite{lucchini-Generators-and-minimal-normal}.

The next theorem is due to Gasch\"utz and it can be used to get a generating set of a group from a generating set of a quotient group. 
\begin{theorem}\cite{gaschutz1956}\label{Gaschütz-thorem}
Let $G$ be a finite group, and let $N$ be a normal subgroup of $G$. Let $g_{1},\ldots,g_{t} \in G$ be such that $G/N=\langle g_{1}N,\ldots,g_{t}N\rangle$. If $G$ can be generated with $t$ element then there exist $x_1,\ldots,x_t \in N$ such that $G=\langle g_1x_1,\ldots,g_tx_t\rangle$.
\end{theorem}

In the above theorem, the number of choices for  $x_1,\ldots,x_t$ is too large. The situation becomes better when the normal subgroup $N$ is abelian, as stated in Theorem~\ref{Generalization-of-lucchini-lemma}. Theorem~\ref{Generalization-of-lucchini-lemma} is essentially due to Lucchini and Menegazzo. However it is not stated exactly as we state in this paper. In the Appendix we prove how this version can be obtained. 

\begin{theorem}\cite{lucchini-minimal-solvable}\label{Generalization-of-lucchini-lemma} Let $N$ be an abelian minimal normal subgroup of a finite group $G$, let $\{e_{1},\dots,e_{m}\}$ be a generating set of $N$ and let \{$g_{1}N,\dots,g_{t}N\}$ be a minimum generating set of $G/N$. If $G$ can also be generated by $t$ elements then either $G=\langle g_{1},\dots,g_{t}\rangle$ or there exist $i$, $1 \leq i \leq t$, and $j, 1 \leq j \leq m$, such that  $\{g_{1},\dots,g_{i-1},g_{i}e_{j},g_{i+1},\dots,g_{t}\}$ is a generating set of $G$. 
\end{theorem}

The following theorem can be used to obtain a minimum generating set of $G$ from a minimum generating set of $G/N$, where $N$ is an abelian minimal normal subgroup. 

\begin{theorem}\cite{lucchini-minimal-solvable,lucchini-Generators-and-minimal-normal}\label{Abelian-Theorem}
Let $G$ be a group and let $N=\langle e_1,\ldots,e_l \rangle$ be an abelian minimal normal subgroup of $G$. If $G/N=\langle g_{1}N,\dots,g_{t}N\rangle$ and $d(G/N)=t$ then 
\begin{itemize}
    \item[(i)] if $d(G)=t$ then either $G=\langle g_{1},\dots,g_{t}\rangle$ or there exists $i$, $1 \leq i \leq t$, and $j, 1 \leq j \leq l$, such that  $G= \langle g_{1},\dots,g_{i-1},g_{i}e_{j},g_{i+1},\dots,g_{t} \rangle$; or
    \item[(ii)] $d(G)=t+1$ and $G=\langle g_{1},\ldots,g_{t},x \rangle$ for any $x\,(\neq1) \in N$.
\end{itemize}
\end{theorem}

The following theorem can be used to get a minimum generating set of $G$ from a minimum generating set of $G/N$, where $N$ is a non-abelian minimal normal subgroup. 

\begin{theorem}\cite{gaschutz1956,lucchini-Generators-and-minimal-normal}\label{Non-Abelian-Theorem}
Let $G$ be a group and let $N$ be a non-abelian minimal normal subgroup of $G$. If $G/N=\langle g_{1}N,\dots,g_{t}N\rangle$ and $d(G/N)=t$ then 
\begin{itemize}
    \item[(i)] if $d(G)= t$ then there exists $x_{1},\ldots,x_{t} \in N$ such that $G=\langle g_{1}x_{1},\ldots,g_{t}x_{t}\rangle$; or
    \item[(ii)] $d(G) = t+1 $ and there exists $x_{1},\ldots,x_{t+1} \in N$ such that either 
    \[G=\langle \{g_{2},\ldots,g_{t}\}\cup\{ x_{1} g_{1},x_{2}\}\rangle\] or
    \[G=\langle g_{1}x_{1},\ldots,g_{t}x_{t},x_{t+1}\rangle.\]  
\end{itemize}
\end{theorem}

\begin{lemma}\label{Gen-Set-Quotient}
Let $G \leq S_n$ be a permutation group and let $A$ be a normal subgroup of  $G$ both given by their generating sets. Let $ \widetilde{N}=\langle h_1A,\ldots,h_lA \rangle$ be the minimal normal subgroup of $G/A$. Let $H$ be the subgroup of $G$ such that $H/A=\widetilde{N}$. Then given a minimum generating set of $G/H$, a minimum generating set of $G/A$ can be found in time $|\widetilde{N}|^{t+1} n^{O(1)}$, where $t=d(G/H)$. Moreover, if $\widetilde{N}$ is abelian then computing a minimum generating set of $G/A$ from a minimum generating set of $G/H$ takes $n^{O(1)}$ time.
\end{lemma}
\begin{proof}
Note that $H= \langle h_1, \ldots,h_l, A \rangle$. From the correspondence theorem  \cite{rotman} it follows that $H$ is normal in $G$. By Jerrum's filter we may assume without loss of generality that $l\leq n-1$. Let $\widetilde{G}=G/A$ and $\widetilde{N}=H/A$. We first consider the case when $G \neq \{e\}$ and $d(G/H)=0$. In this case $H=G$ and $\widetilde{N}= \widetilde{G}$. Then a minimal normal subgroup of $\widetilde{G}$ is $\widetilde{G}$ itself. Thus, $\widetilde{G}$ is simple. If $\widetilde{G}$ is abelian then any non-identity element of $\widetilde{G}$ will form a minimum generating set of $\widetilde{G}$. If $\widetilde{G}$ is non-abelian then we pick a fixed non-identity element $x \in \widetilde{G}$ and try all $ y \in \widetilde{G}$. By the result of Guralnick and Kantor  \cite{Kantor-simple-group}, one of the choices of $y$ along with $x$ generates $\widetilde{G}$ and therefore form a minimum generating set of $\widetilde{G}=G/A$.

Now we assume $d(G/H)\neq 0$. Define $\phi: G/H \rightarrow \widetilde{G}/\widetilde{N}$ such that $\phi(gH)=(gA)\widetilde{N}$. One can check that $\phi$ is an isomorphism between $G/H$ and $\widetilde{G}/\widetilde{N}$. Therefore, if $\{ g_1H, g_2H,\ldots,g_t H \}$ is a minimum generating set of the group $G/H$, then $\{ (g_1A)\widetilde{N},\ldots,(g_tA)\widetilde{N} \}$ is a minimum generating set of $\widetilde{G}/\widetilde{N}$. Now, we find a minimum generating set of $\widetilde{G}$.

We know that $\widetilde{N}$ is a minimal normal subgroup of $\widetilde{G}$ and $d(\widetilde{G}/\widetilde{N})=t$ with $\widetilde{G}/\widetilde{N}=\langle (g_1A)\widetilde{N}, \ldots,(g_tA)\widetilde{N} \rangle$. By Theorem \ref{bound-on-d(G)} either $d(\widetilde{G})=t
$ or $d(\widetilde{G})=t+1$. If $d(\widetilde{G})=t$ then by Theorem \ref{Gaschütz-thorem}, there exists $x_{1}A,\ldots,x_{t}A \in \widetilde{N}$ such that $\widetilde{G}=\langle g_{1}x_{1}A,\ldots,g_{t}x_{t}A \rangle$. If $d(\widetilde{G})=t+1$ then by the proof of Theorem \ref{bound-on-d(G)}, there exists $x_{1}A,\ldots,x_{t}A,x_{t+1}A \in \widetilde{N}$ such that either $\widetilde{G}=\langle g_{1}x_{1}A,\ldots,g_{t}x_{t}A,x_{t+1}A\rangle$ or $\widetilde{G}=\langle x_{1}g_{1}A,g_{2}A,\ldots,g_{t}A,x_{2}A\rangle$. Thus it takes at most $|\widetilde{N}|^{t+1} n^{O(1)}$ time to find a minimum size generating set of $\widetilde{G}$. 

Assume that $\widetilde{N}$ is abelian. In this case if $d(\widetilde{G})=t$ then by Theorem \ref{Generalization-of-lucchini-lemma} and Theorem \ref{Abelian-Theorem} either $\{ g_{1}A,\ldots,g_{t}A\}$ generates $\widetilde{G}$ or there exists $1 \leq i \leq t$ and $1 \leq j \leq l$ such that $\{ g_{1}A,\ldots,g_{i-1}A,g_{i}h_{j}A,g_{i+1}A,\ldots,g_{t}A\}$ is a generating set of $\widetilde{G}$. If $d(\widetilde{G})=t+1$ then by Theorem \ref{Abelian-Theorem}, we know that for any $1 \neq xA \in \widetilde{N}$ then $\{ g_{1}A,\ldots,g_{t}A,xA\}$ is a generating set of $\widetilde{G}$ of minimum size. Notice that, $l,t \leq (n-1)$. Thus we can find a generating set of $\widetilde{G}$ of minimum size by checking all possible such sets which takes at most $n^{O(1)}$ time. Thus, in this case we can find a generating set of $\widetilde{G}$ in polynomial time.
\end{proof}

The above proof can be used to prove an analogous theorem for groups given by their Cayley table.

\begin{lemma}\label{Gen-Set-Quotient-Cayley}
Let $G$ be a group of order $n$ given by its Cayley table and $A \trianglelefteq G$. Let $ \widetilde{N}=\langle h_1A,\ldots,h_lA \rangle$ be a minimal normal subgroup of $G/A$.  Let $H$ be the subgroup of $G$ such that $H/A=\widetilde{N}$. Then given a minimum generating set of $G/H$, a minimum generating set of $G/A$ can be found in time $|\widetilde{N}|^{t+1} n^{O(1)}$, where $t=d(G/H)$. Moreover, if $\widetilde{N}$ is abelian then computing a minimum generating set of $G/A$ from a minimum generating set of $G/H$ takes $n^{O(1)}$ time.
\end{lemma}

\section{General Groups in the Cayley Table Representation}\label{MIN-GEN-General-Groups}

In this section, we design an algorithm for finding a minimum generating set of any group given by its Cayley table representation. Recall that in this Cayley table representation of a group, $n$ is the order of a group.


Notice that by trying all possible sets, it is easy to obtain an $n^{\log_{p} n+O(1)}$ time algorithm for the MIN-GEN problem. We present a significantly improved algorithm for the minimum generating set problem for general groups and prove that there is an $n^{(1/4) \log_{p} n+O(1)}$ time algorithm that finds a minimum generating set of an input group given by its Cayley table. 

\begin{lemma}\label{bounds on total gene set-GeneralGroup}
Let $G$ be a group of order $n$, and let $N$ be a minimal normal subgroup of $G$. Suppose that $d(G/N)=t$ and $G/N = \langle g_{1}N,  \ldots,g_{t}N \rangle$. Then $|N|^{t+1} \leq n^{(1/4) \log_{p} n+1}$.
\end{lemma}

\begin{proof} Since $N$ is a minimal normal subgroup of $G$, we have $|N|=n^{\epsilon}$ for some $0 < \epsilon \leq 1$ (since every minimal normal subgroup is nontrivial subgroup). Thus $|G/N|= n^{1-\epsilon}$ and $\log_{p} |G/N| = (1-\epsilon) \log_{p} n$. Therefore, we get $t=d(G/N) \leq (1-\epsilon) \log_{p} n$. Then $|N|^{t+1} \leq n^{\epsilon (1-\epsilon) \log_{p} n+1}$. Since $0 < \epsilon \leq 1$, the maximum possible value of $\epsilon (1-\epsilon)$ can be at most $1/4$. Hence, we get $|N|^{t+1} \leq  n^{(1/4) \log_{p} n+1}$.
\end{proof}

\begin{theorem}\label{thm-quotient-general-group}
Let $G$ be a group of order $n$ given by its Cayley table and let $A \trianglelefteq G$. Then a minimum generating set of $G/A$ can be found in $ n^{(1/4) \log_{p} n+ O(1)}$ time.
\end{theorem}

\begin{proof}
If $G/A=\{e \}$ then the problem is trivial. We first compute a minimal normal subgroup $\widetilde{N}$ of $G/A$. This could be done in polynomial time \cite{ronyai-minimal-normal-algo,Ako-Seress}. Notice that, $\widetilde{N}$ will be of the form $H/A$, where $H =\langle h_1, \ldots, h_l, A \rangle$ is a normal subgroup of $G$. Next we recursively find a minimum generating set of $G/H$. Using Lemma \ref{Gen-Set-Quotient-Cayley}, a minimum generating set of $G/A$ could be found in time $|\widetilde{N}|^{t+1} n^{O(1)}$, where $t=d(G/H)$. From Lemma \ref{bounds on total gene set-GeneralGroup} we have, $|\widetilde{N}|^{t+1} n^{O(1)} \leq n^{(1/4) \log_{p} n+ O(1)}$. 

Notice that in each recursive call the size of $G/A$ reduces by at least half. Thus the number of recursive call is $O(\log n)$. Therefore, the total running time of the algorithm is $ n^{(1/4) \log_{p} n+ O(1)}$.
\end{proof}

\begin{theorem}
Let $G$ be a group of order $n$ given by its Cayley table. Then a minimum generating set of $G$ can be found in $ n^{(1/4) \log_{p} n+ O(1)}$ time.
\end{theorem}

\begin{proof}
Take $A=\{e\}$ in Theorem \ref{thm-quotient-general-group}.
\end{proof}

\section{MIN-GEN for Product of Simple Groups and its Consequences}\label{MIN-GEN-Product-Simple-Groups}

In this section, we consider groups that are  direct products of non-abelian simple groups. The isomorphism of such groups could be checked in polynomial time \cite{Kayal-Neeraj,Wilson-complement-algo-PG,babai-groupIso}. We mention that the decomposition of a group into (indecomposable) direct factors is known to be computable in polynomial time even if the direct factors are not simple \cite{Kayal-Neeraj,Wilson-complement-algo-PG}. 

Let $ \mathcal{G}_{\Pi simp}=\{ G \,|\, G \text{ is a direct product of simple groups} \}$. We show that a minimum generating set of a group $G \in \mathcal{G}_{\Pi simp}$ can be found in polynomial time.

\begin{theorem}\label{product-simple-groups-proof}
Let $G \in \mathcal{G}_{\Pi simp}$ be given by its Cayley table. Then a minimum generating set of $G$ can be computed in $n^{O(1)}$ time, where $n= |G|$.
\end{theorem}

\begin{proof}
Let $G=S_1 \times \cdots \times S_r$ where $S_i$'s are simple groups. The decomposition can be computed in polynomial time \cite{Kayal-Neeraj}. We design an algorithm that works in $r$ stages. In the $i$th stage it computes a minimum generating set of $S_1 \times  \cdots \times S_{i}$. Note that $S_i$ is a minimal normal subgroup of $S_1 \times  \cdots \times S_{i}$. Therefore, we can apply Theorem \ref{Abelian-Theorem} if $S_i$ is abelian or Theorem \ref{Non-Abelian-Theorem} if $S_i$ is non-abelian. An application of Theorem \ref{Non-Abelian-Theorem} corresponds to trying $|S_{i}|^{t+1}$ many subsets, where $t=d(S_1 \times  \cdots \times S_{i-1})$. How do we ensure that $|S_{i}|^{t+1}$ is a polynomial in $n$ ? The idea is to reorder the simple groups such that $|S_1| \geq \cdots \geq |S_r|$. Note that $d(S_1) \leq 2$ as any simple groups can be generated by at most $2$ elements. Observe that in each stage the size of newly computed minimum generating set increases by at most $1$ (see Theorem \ref{bound-on-d(G)}). Therefore, $t=d(S_1 \times  \cdots \times S_{i-1}) \leq i$ for $i>1$.

Computing a minimum generating set of $S_1$ takes time $|S_1|^2 n^{O(1)}$. Fix $i > 1$, 
\begin{align*}
|S_i|^{i-1} & \leq \prod_{j=1}^{i-1} |S_j|\\
|S_i|^{i+1} & \leq |S_i| \prod_{j=1}^{i} |S_j| \leq \Big(\prod_{j=1}^{i} |S_j|\Big)^2 \leq n^2.
\end{align*}

Thus the maximum number of subsets tried in stage $i$ is $|S_{i}|^{t+1} \leq |S_{i}|^{i+1} \leq n^2$. As $r \leq \log n$ the runtime of the algorithm is $n^{O(1)}$.
\end{proof}



As a consequence of Theorem \ref{product-simple-groups-proof}, we show how to solve the following problem: Given a group $G$ by its Cayley table along with subgroups $G_1 $ and $G_2$ such that $G=G_1 \times G_2$ and $d(G_1)$ and $d(G_2)$ find $d(G_1 \times G_2)$.

We note that $d(G)$ does not solely depend on \emph{just the numbers} $d(G_1)$  and $d(G_2)$. In fact, the example given in Subsection \ref{Product-Simple-Intro} for $A_5$, $A_5^{19}$ and $A_5^{6464040}$ shows that the relation between $d(G)$ and the factors of $G$ could be intricate.

\begin{theorem}\label{product-of-groups}
    Let $G=G_1 \times \ldots \times G_r$ be a group given by its Cayley table along with $d(G_i)$ for each $i \in [r]$. Then $d(G)$ can be computed in polynomial time.
\end{theorem}
\begin{proof}
It is enough to prove the theorem for $r=2$. Let $J$ be the intersection of all maximal normal subgroups of the given group $G$. The subgroup $J$ can be computed in polynomial time even for permutation group \cite{Ako-Seress}. Collins prove that the quotient $G/J$ is a product of simple groups (see e.g., p. 16 \cite{collins-thesis}). They also showed that $d(G_1 \times G_2)=\max\{d(G_1), d(G_2), d(G_1 \times G_2/J)\}$. 
Now it is clear from Theorem \ref{product-simple-groups-proof}, that $d(G_1 \times G_2)$ can be computed in polynomial time.
\end{proof}




\begin{corollary}
    Given a group $G$ by its Cayley table the MIN-GEN problem polynomial time Turing reduces to the MIN-GEN problem for indecomposable groups. 
\end{corollary}

\begin{proof}
    It is not hard to prove that given an oracle for MIN-GEN, we can find $d(G)$ of a given group $G$ by querying MIN-GEN multiple times ($O(\log \log n)$ queries are enough).

    The first step in the reduction is to find all the   indecomposable factors of the input group $G$ \cite{Kayal-Neeraj,Wilson-complement-algo-PG}. Let $G_1, G_2, \ldots, G_r$ be the indecomposable factors. Using the oracle for MIN-GEN problem for indecomposable groups we can find $d(G_i)$ for $i \in [r]$. Now, the corollary follows from Theorem \ref{product-of-groups}.
\end{proof}

\section{Solvable Permutation Groups}\label{MIN-GENProblem-Solvable-PG}
In this section, we design a simple recursive polynomial-time algorithm to compute a minimum generating set of solvable permutation groups based on the ideas by Gasch\"utz \cite{gaschutz1956}, and Lucchini and Menegazzo \cite{lucchini-minimal-solvable}. 

 
For all permutation group problems, an input group $G \leq S_n$ is given by a generating set of $G$. Recall that $n$ is now the degree of a group $G$.  

\begin{lemma}\label{Solvable-PG-Lemma}
Let $G \leq S_{n}$ be a solvable group and $A$ be a normal subgroup of $G$ given by their generating sets. Then a minimum generating set of $G/A$ could be computed in polynomial time.
\end{lemma}

\begin{proof}
If $G/A=\{e \}$ then the problem is trivial. We first compute a minimal normal subgroup $\widetilde{N}$ of $G/A$. This could be done in polynomial time \cite{ronyai-minimal-normal-algo}. The algorithm for finding a minimal normal subgroup works for quotient groups too (see e.g., \cite{Ako-Seress,kantor-computing-quotient}). Note that, $\widetilde{N}$ will be of the form $H/A$, where $H =\langle h_1, \ldots, h_l, A \rangle$ is a normal subgroup of $G$. Next we recursively find a minimum generating set of $G/H$. Since $G$ is solvable $\widetilde{N}$ will be an abelian subgroup of $G$. This follows from the fact that minimal normal subgroups are product of isomorphic simple groups (see e.g., \cite{rotman}). Using Lemma \ref{Gen-Set-Quotient}, a minimum generating set of $G/A$ could be found in polynomial time.

Since in each subsequent recursive calls the size of $G/A$ reduces by at least half, the number of iteration is $O(\log n!)=O(n \log n)$. Therefore, the runtime of the algorithm is polynomial in $n$. 
\end{proof}

If we take $A=\{e\}$ in Lemma \ref{Solvable-PG-Lemma}, we obtain the following result. 
\begin{theorem}(\cite{gaschutz1959mingensolvable,lucchini-minimal-solvable})\label{Solvable-PG-Theorem}
Let $G \leq S_{n}$ be a solvable permutation group given by its generating set. Then a minimum generating set of $G$ can be computed in polynomial time.
\end{theorem}

Now we discuss a problem similar to the one discussed at the end of Section~\ref{MIN-GEN-Product-Simple-Groups}. Namely, the problem of computing a minimum generating set of a permutation group that is a direct product of two groups. As discussed before, even if we know a minimum generating set of groups $M$ and $K$, it is not clear how to compute a minimum generating set of $M\times K$ or $d(M \times K)$. There are two challenges in applying the approach discussed at the end of Section \ref{MIN-GEN-Product-Simple-Groups}. Firstly, we do not know how to solve the MIN-GEN problem for permutation groups which are products of simple groups. The other issue is working with quotient groups\footnote{The Cayley table of a quotient group, say $G/J$, can be computed in polynomial time if the Cayley tables of $G$ and $J$ are known. For permutation groups, working with a quotient group needs extra effort as the degree of a quotient group could be very large. While we believe that for the current problem, the issue of a quotient group can be taken care of, we are not completely certain.}.

The situation, however, is different when at least one of the product groups is solvable. Formally, let $G=M \times K$ where $M$ is a solvable permutation group and $K$ is a permutation group such that a minimum generating set $\{ k_1,\ldots ,k_r \}$ of $K$ is already known. In this case, we can design a recursive algorithm to find a minimum generating set of $G$.

\begin{lemma}\label{Product-PG-Lemma}
Let $G=M \times K \leq S_{n}$ where  $M$ is a solvable permutation group and $K$ is any permutation group such that a minimum generating set  $\{ k_1,\ldots ,k_r \}$ of $K$ is given. Let $A \times \{e\}$ be a normal subgroup of $G$. Then a minimum generating set of $G/A \times  \{e\}$ could be found in polynomial time.
\end{lemma}
\begin{proof}
The proof is similar as the proof of Lemma \ref{Solvable-PG-Lemma}. The detailed proof is given in the Appendix (Lemma \ref{Product-PG-Lemma-Appendix}).
\end{proof}

\begin{theorem} 
Let $G=M \times K \leq S_{n}$ where $M$ is a solvable permutation group and $K$ be a permutation group such that a minimum generating set $\{ k_1,\ldots ,k_r \}$ of $K$ is given. Then a minimum generating set of $G$ can be computed in polynomial time. 
\end{theorem}

\begin{proof}
In Lemma \ref{Product-PG-Lemma}, take $A=\{e\}$.
\end{proof}

\section{The MIN-GEN Problem for Primitive Permutation Groups - A Quasi-Polynomial Time Algorithm}\label{MIN-GEN-Primitive-PG}
In this section, we describe an algorithm to solve the MIN-GEN problem for primitive permutation groups that runs in quasi-polynomial time in the degree of the input group.

We need the following three theorems to prove our main result of this section. 
\begin{theorem}\cite{lucchini-bounds-primitived(G)}\label{bounds on d(G) for primitive PG}
There exists a constant $c$ such that if $G$ is a primitive permutation group of degree $n \geq 3$ then $$d(G) \leq \frac{c \log n}{\sqrt{\log \log n}}.$$
\end{theorem}

\begin{theorem}\cite{lucchini-uniuqe-minimal}\label{unique-minimal-thm}
If $G$ is a non-abelian finite group with a unique minimal normal subgroup $N$ then $d(G)= \max(2,d(G/N))$.
\end{theorem}

\begin{theorem}(see e.g., \cite{dixon-permutation,liebeck1984})\label{Base-order-primitive-permutation}
Let $G$ be a primitive permutation group of degree $n$. Then there is a constant $b > 0$ such that at least one of the following holds:
\begin{itemize}
    \item[(i)] there are positive integers $l$, $k$, and $m$ such that $G$ has a socle which is permutation isomorphic to $A_m^l$ where the action of $A_m$ is equivalent to its action on $k$-element subsets of $\{1,\ldots,m \}$ and $n=\genfrac(){0pt}{2}{m}{k}^l$; or
\item[(ii)] $G$ has order less than $exp(b(\log n)^2)$.
\end{itemize}
\end{theorem}

We now state and prove our result on the minimum generating set problem for primitive permutation groups. 

\begin{theorem}\label{MinGen for Primitive-PG}
Let $G\leq S_n$ be a primitive permutation group. Then  $d(G)$ can be computed in time $exp(\log^{O(1)}(n))$.
\end{theorem}
\begin{proof}

Given a permutation group $G$, if $|G| \leq exp(b(\log n)^2)$ then Theorem \ref{bounds on d(G) for primitive PG} implies that a minimum generating set of $G$ can be found in quasi-polynomial time.

Thus it is enough to consider the case when $G$ is a primitive group and there are positive integers $l$, $k$, and $m$ such that $G$ has a socle which is permutation isomorphic to $A_m^l$ where the action of $A_m$ is equivalent to its action on $k$-element subsets of $\{1,\ldots,m \}$ and $n=\genfrac(){0pt}{2}{m}{k}^l$. Let $H:=  \mathrm{soc}(G)$.

We note that $H \cong A_m^l$ and it is not regular (see  page no. 137, \cite{dixon-permutation}). Since $ n=\genfrac(){0pt}{2}{m}{k}^l$, $G$ is of product type (see Section 4.8 of \cite{dixon-permutation}) and $H$ is the unique minimal normal subgroup of $G$ (see Theorem 4.3B, \cite{dixon-permutation}). Note that a minimal normal subgroup of $G$ (or here the socle of $G$) can be found in polynomial time (e.g., see page no 49, \cite{Ako-Seress}). Thus, we have $d(G)=\max(2,d(G/H))$. 

Now we prove that $|G/H| \leq n^2 (\log n)^{\log n}$.

Since $H$ is not regular, the centralizer $C_{G}(H)=1$ (see e.g., Theorem 4.3B, \cite{dixon-permutation}). In this case it is easy to check that the conjugation action of $G$ on $H$ gives an embedding of $G$ into $\mathrm{Aut}(H)$. 
Let $\Gamma=\{1,2,\ldots,l\}$. 
\begin{align*}
    |G/H| &\leq |\mathrm{Aut}(H)|/|H|\\
        &= |\mathrm{Aut}(H)|/ |\mathrm{Inn}(H)| \hspace{34pt} (\text{as}\,\, \mathrm{Inn}(H) \cong H/Z(H)\, \text{and} \, Z(H)=1)\\
        &=|\mathrm{Out(H)}|= |\mathrm{Out(A_m^{l})}| \\
        &= |\mathrm{Out}(A_m) \wr_{\Gamma} S_l| \hspace{41pt} (\text{see page no. 131 \cite{dixon-permutation}})\\
        &\leq4^{l} l! \hspace{71pt}(\text{as} \,\, |\mathrm{Out}(A_m)|=2\,\, \text{when}\, m \neq 6 \,\text{and} \,|\mathrm{Out}(A_6)|=4)
\end{align*}
Notice that, $n\geq m^l$ which implies that $l \leq \log n$ (since $m > 2$). Thus, $4^l l!\leq n^2 (\log n)^{\log n}$.

Therefore, we can find a minimum generating set of $G/H$ and find $d(G)$ using Theorem \ref{unique-minimal-thm} in quasi-polynomial time. Note that the final output is the size of the minimum generating set and not a generating set. 
\end{proof}
The following corollary is a direct application of Theorem \ref{MinGen for Primitive-PG}. We consider the wreath product $H \wr G$ of a solvable group $H$ and a primitive permutation group $G$ such that  $\mathrm{gcd}(|H/H'|,|G|)=1$ and show that $d(H \wr G)$ can be found in quasi-polynomial time.

\begin{corollary}
Let $H$ and $G$ be two subgroups of $S_n$. Suppose, $H$ is a solvable permutation group and  $G$ is a primitive permutation group  such that $\mathrm{gcd}(|H/H'|,|G|)=1$. Then $d(H \wr G)$ can be found in quasi-polynomial time, where $H'$ is the commutator subgroup of $H$.
\end{corollary}
\begin{proof}
Since $H/H'$ is an abelian and $\mathrm{gcd}(|H/H'|,|G|)=1$ we have $d((H/H') \wr G)= \max\{d(H/H')+1, d(G)\}$ (see e.g., Corollary 6, \cite{lucchini-wreath-product-1997generating}). 
Using Theorem \ref{Solvable-PG-Theorem} and \ref{MinGen for Primitive-PG} we can compute $d((H/H') \wr G)$ in quasi-polynomial time. Also it is known that 
$d(H \wr G)= \max \{ d((H/H') \wr G, \frac{d(H)-2}{n}+2\}$  (see e.g., Theorem 2, \cite{lucchini-wreath-product-1997generating}). 
\end{proof}

\section{Permutation Groups with Bounded Non-abelian Chief Factors}\label{Min-Gen-Bounded-non-abelain-chief-factor}
In this section, we design a $\mathrm{DTIME}(2^n)$ time algorithm to solve the minimum generating set problem for groups with bounded non-abelian chief factors where $n$ is the degree of the input permutation group. The algorithm is based on the ideas by Lucchini and Menegazzo \cite{lucchini-minimal-solvable}.

Note that the class ${\Gamma}(l)$ of groups whose composition factors can be embedded in $S_l$ was considered by Babai, Cameron, and P\'alfy in the context of composition factor \cite{babai-palfy-cameron}. Note that the composition factor of groups in $\Gamma(l)$ are of bounded size for constant $l$. The class $\chi(l)$, defined below, is a  natural analog of ${\Gamma}(l)$ in the context of chief factors. It is defined as follows: 

\begin{definition}
For a positive integer $l$, let $\chi(l)$ denote the class of all permutation groups each of whose non-abelian chief factors are of order at most $l$, i.e.,
$$\chi(l)=\{G \,|\text{ every non-abelian chief factor of } G \text{ has order at most }l\}.$$
\end{definition}

\begin{theorem}\label{general-bound-X(l)}
Let $G \leq S_n$ be a group in the class $\chi(l)$. Then there is an algorithm to find a minimum generating set of $G$ that runs in time $l^{d(G)+1} n^{O(1)}$. 
\end{theorem}

\begin{proof}
The first step in the algorithm is to compute a chief series of $G$: $$ 1= G_{m} \trianglelefteq  G_{m-1} \trianglelefteq  \dots \trianglelefteq G_1 \trianglelefteq G_{0}=G.$$

It is known that a chief series can be computed in polynomial time \cite{Ako-Seress}.

The algorithm now proceeds in stages. In the $1$st stage the algorithm computes a minimum generating set of $G/G_{1}$. Note that $G/G_1$ is a chief factor and it is also a simple group. If $G/G_1$ is abelian then any non-identity element will be a generator of this group. If $G/G_1$ is non-abelian then $|G/G_1| \leq l$ as $G \in \chi(l)$. We also know that a minimum generating set of any non-abelian simple group is of size $2$. Therefore, a minimum generating set of $G/G_1$ could be found in time $l^2 n^{O(1)}$ by trying all possible two elements subsets of $G/G_1$.

In the $i$th stage, the algorithm computes a minimum generating set of $G/G_i$ assuming that a minimum generating set of $G/G_{i-1}$ has already been computed in the previous stage. 

Let $\{g_1G_{i-1}, \ldots, g_tG_{i-1}\}$ be a minimum generating set of $G/G_{i-1}$. Notice that $G_{i-1}/G_{i}$ is a minimal normal subgroup of $G/G_{i}$ and $(G/G_{i})/(G_{i-1}/G_{i}) \cong G/G_{i-1}$. Now we apply Lemma \ref{Gen-Set-Quotient} with $A=G_{i}$, $\widetilde{N}=G_{i-1}/G_{i}$ and $H=G_{i-1}$ to get a minimum generating set of $G/A=G/G_{i}$. If the chief factor $\widetilde{N}=G_{i-1}/G_{i}$ is abelian then the $i$th stage takes polynomial time. Otherwise if the chief factor is non-abelian then a minimum generating set of $G/A=G/G_{i}$ can be found in time $|\widetilde{N}|^{t+1} n^{O(1)}$. But $|\widetilde{N}| \leq l$ as $\widetilde{N}$ is a chief factor of $G \in \chi(l)$. Therefore, the $i$th stage takes time $l^{t+1} n^{O(1)}$. Since $d(G/G_{i}) \leq d(G)$ for all $i$. Each stage takes time $l^{d(G)+1} n^{O(1)}$. As $G/G_m \cong G$, in the $m$th stage the algorithm outputs a minimum generating set of $G$. Since the number of stages is at most $O(n \log n)$ the total running time is $l^{d(G)+1} n^{O(1)}$.
\end{proof}

\begin{remark}
    The above algorithm is FPT with respect to the parameters $(l,d(G))$.
\end{remark}

An easy application of Theorem \ref{general-bound-X(l)} shows that a minimum generating set of a permutation groups in $\chi(l)$ is in $\mathrm{DTIME}(2^{O(n)})$.

\begin{corollary}
The minimum generating set problem for permutation groups in $\chi(l)$ is in $\mathrm{DTIME}(2^{O(n)})$.
\end{corollary}
\begin{proof}
This follows from the fact that any permutation group has a generating set of size at most $n-1$ \cite{jerrum1986compact}.
\end{proof}

\begin{theorem}
Let $G\in \chi(l)$ be such that $G$ is transitive, then there is a sub-exponential algorithm with runtime  $l^{o(n)+O(1)} n^{O(1)}$  to find a minimum generating set of $G$.
\end{theorem}

\begin{proof}
There exists a constant $c >0$ such that for all transitive permutation group $G$, $d(G) \leq cn/(\sqrt{\log n})$ \cite{transitive-pg-bound}. Thus Theorem \ref{general-bound-X(l)} implies that there is an algorithm to find a minimum generating set of $G$ in time $l^{o(n)+O(1)} n^{O(1)}$.
\end{proof}

\section{Acknowledgement}
The authors would like to thank V. Arvind for discussion and his valuable inputs on the problems.

\bibliographystyle{plainurl}

\bibliography{samplepaper}



\newpage

\section{Appendix}
We now prove Theorem \ref{bound-on-d(G)} when $N$ is abelian. 

\begin{theorem}\cite{lucchini-Generators-and-minimal-normal}\label{bound-on-d(G)-abelian-case}
If $G$ is a finite group and $N$ is an abelian minimal normal subgroup of G, then $d(G)\leq max(2,d(G/N)+1)$. In particular, $d(G/N)\leq d(G) \leq d(G/N)+1$.
\end{theorem}

\begin{proof}
If $G=N$ then $G$ is simple and $d(G) \leq 2$. 

Let  $\{g_{1}N,\ldots,g_{t}N \}$ be a minimum generating set of $G/N$. We show that for every $x\,(\neq1) \in N$ the set $\{g_{1}, \ldots,g_{t},x\}$ generates $G$. Fix $x (\neq 1) \in N$. Consider the normal closure of $x$ i.e., $\langle \{x\}^G \rangle = \langle \{g^{-1}xg | g \in G \} \rangle$. Since $ 1\neq \langle \{x\}^G \rangle \leq N$, $\langle \{x\}^G \rangle \trianglelefteq G $ and $N$ is a minimal normal subgroup of $G$ we have $\langle \{x\}^G \rangle=N$. Let $g \in G $, then $g=n g_{i_1}\cdots g_{i_p}$ for some $n \in N$ and $1 \leq i_1,i_2,\ldots i_p \leq t$. Therefore it is enough to show that any $n\in N$ can be generated using the elements $g_1,\ldots,g_t$ and $x$. Since $n \in N$ we get $n=h^{-1}xh $ for some $h \in G$. There exists $n' \in N$ and $1 \leq j_1,\ldots,j_r \leq t$ such that $h=n' g_{j_1}\cdots g_{j_r}$. Since $N$ is abelian we obtain, $n=(n' g_{j_1}\cdots g_{j_r})^{-1}x(n' g_{j_1}\cdots g_{j_r})=( g_{j_1}\cdots g_{j_r})^{-1} x (g_{j_1}\cdots g_{j_r})$. Thus, $G =\langle \{g_{1},\ldots,g_{t},x\}\rangle$.
\end{proof}

We now give the proof of Theorem~\ref{Generalization-of-lucchini-lemma} which can be proved using the following lemma.

\begin{lemma}\cite{lucchini-minimal-solvable}\label{Lucchini-extension-lemma} Let $N$ be an abelian normal subgroup of a group $G$, suppose that $\{e_{1},\dots,e_{m}\}$ is a  generating set of $N$ and $G/N=\langle g_{1}N,\dots,g_{t}N\rangle$. If $\langle g_{1},\dots,g_{t}\rangle \cap N = 1$, but there exists $u_{1},\dots,u_{t} \in N$ such that $\langle g_{1}u_{1},\dots,g_{t}u_{t}\rangle \cap N \neq 1$, then there exists $i$, $1 \leq i \leq d$, and $j, 1 \leq j \leq m$, such that  \[\langle g_{1},\dots,g_{i-1},g_{i}e_{j},g_{i+1},\dots,g_{t} \rangle \cap N \neq 1.\]
\end{lemma}

\noindent \textbf{Proof of Theorem~\ref{Generalization-of-lucchini-lemma}:}
Assume that $\langle g_{1},\dots,g_{t}\rangle \cap N \neq 1$. Let  $y$ be a non-identity element  in  $\langle g_{1},\dots,g_{t}\rangle \cap N$. Since $N$ is abelian Theorem \ref{bound-on-d(G)} implies that for every $x\,(\neq1) \in N$, $\{g_{1}, \ldots,g_{t},x\}$ is a generating set of $G$. In particular, $\langle g_{1},\dots,g_{t},y\rangle=G$. Since $y \in \langle g_{1},\dots,g_{t}\rangle$, $G= \langle g_{1},\dots,g_{t}\rangle$.

Let $\langle g_{1},\ldots,g_{t}\rangle \cap N = 1$. Since $G$ can be generated by $t$ elements, by Theorem \ref{Gaschütz-thorem}, there exists $u_{1},\ldots,u_{t} \in N$ such that $\langle g_{1}u_{1},\ldots,g_{d}u_{t}\rangle = G$. Moreover, $G \cap N= \langle g_{1}u_{1},\ldots,g_{t}u_{t}\rangle \cap N = N \neq 1$. We can now apply Lemma \ref{Lucchini-extension-lemma}: $\langle g_{1},\dots,g_{t}\rangle \cap N = 1$, but there exist $u_{1},\dots,u_{t} \in N$ such that $\langle g_{1}u_{1},\dots,g_{t}u_{t}\rangle \cap N \neq 1$. Therefore, there exists $i$, $1 \leq i \leq t$, and $j, 1 \leq j \leq m$, such that  $\langle g_{1},\dots,g_{i-1},g_{i}e_{j},g_{i+1},\dots,g_{t} \rangle \cap N \neq 1$. \qed

\vspace{.2cm}

Next, we prove Lemma \ref{Product-PG-Lemma}

\begin{lemma}\label{Product-PG-Lemma-Appendix}
Let $G=M \times K \leq S_{n}$ where  $M$ is a solvable permutation group and $K$ is any permutation group such that a minimum generating set  $\{ k_1,\ldots ,k_r \}$ of $K$ is given. Let $A \times \{e\}$ be a normal subgroup of $G$. Then a minimum generating set of $G/A$ could be found in polynomial time.
\end{lemma}
\begin{proof}
Here too we design a recursive algorithm similar to the one in Lemma \ref{Solvable-PG-Lemma}. The base case is when $M/A=\{e\}$ which can be solved easily.  

Otherwise we find a minimal normal subgroup $\widetilde{N}_{1}$ of $M/A$. $\widetilde{N}_{1}$ will be of the form $H/A$ for some normal subgroup $H=\langle h_1,\ldots,h_l \rangle$ of $M$. It is easy to check that $\widetilde{N} = (H \times \{e\}) / (A \times \{e\})$ is a minimal normal subgroup of $(M \times K) / (A \times \{e\})$. Next we recursively compute a minimum generating set of $(M \times K) / (H \times \{e\})$. As before $\widetilde{N}$ will be abelian and we can use Lemma \ref{Gen-Set-Quotient} to find a minimum generating set of $(M \times K) / (A \times \{e\})$ in polynomial time. The total running time  is polynomial as the number of recursive call is $O(n \log n)$.  
\end{proof}
\end{document}